\documentclass[a4paper,12pt]{amsart}
\usepackage{amsmath,amsthm,amssymb}
\usepackage{hyperref}

\DeclareMathOperator{\RE}{Re}
\DeclareMathOperator{\IM}{Im}

\allowdisplaybreaks[1]

\newcommand{\ti}{\widetilde}
\newcommand{\la}{\lambda}

\newcommand{\ze}{\zeta}
\newcommand{\ity}{\infty}
\newcommand{\C}{\mathbb{C}}

\newcommand{\N}{\mathbb{N}}
\newcommand{\Z}{\mathbb{Z}}
\newcommand{\B}{\Big}
\newcommand{\F}{\mathfrak{F}}
\newcommand{\G}{\mathfrak{G}}

\numberwithin{equation}{section}
\newtheorem{theorem}{Theorem}[section]
\newtheorem{lemma}[theorem]{Lemma}
\newtheorem{corollary}[theorem]{Corollary}
\theoremstyle{remark}
\newtheorem{remark}[theorem]{Remark}
\newtheorem{example}[theorem]{Example}
\newtheorem{definition}[theorem]{Definition}

\thanks {The research work of the first  author is supported by research fellowship from Council of Scientific and Industrial Research (CSIR), New Delhi.}

\begin{document}
\title[semigroups of transcendental functions]{the dynamics of semigroups of transcendental entire functions II}
\author[D. Kumar]{Dinesh Kumar}
\address{Department of Mathematics, University of Delhi,
Delhi--110 007, India}

\email{dinukumar680@gmail.com }
\author[S. Kumar]{Sanjay Kumar}

\address{Department of Mathematics, Deen Dayal Upadhyaya College, University of Delhi,
New Delhi--110 015, India }

\email{sanjpant@gmail.com}

\begin{abstract}
We introduce the concept of escaping set for semigroups of transcendental entire functions using Fatou-Julia theory. Several results of the  escaping set associated with the iteration of one transcendental entire function have been extended to transcendental semigroups. We also investigate the properties of escaping sets for conjugate semigroups and abelian transcendental semigroups. Several classes of transcendental semigroups for which Eremenko's conjecture holds have been provided.
\end{abstract}



\maketitle

\section{introduction}\label{sec1}

 Let $f$ be a transcendental entire function. For $n\in\N$ let $f^n$ denote the $n$-th iterate of $f.$ The set $F(f)=\{z\in\C : \{f^n\}_{n\in\N}\,\text{ is normal in some neighborhood of}\, z\}$ is called the Fatou set of $f,$ or the set of normality of $f$ and its complement $J(f)$ is called the Julia set of $f$. The escaping set of $f$ denoted by $I(f)$ is the set of points in the complex plane that tend to infinity under iteration of $f$. In general, it is neither an open nor a closed subset of $\C$ and has interesting topological properties.  For a transcendental entire function $f,$ the escaping set was studied for the first time by Eremenko in \cite {e1} who proved that
\begin{enumerate}
\item\ $I(f) \neq\emptyset;$
\item\ $J(f)=\partial I(f);$
\item\ $I(f)\cap J(f)\neq\emptyset;$
\item\ $\overline{I(f)}$ has no bounded components.
\end{enumerate}
In the same paper he stated the following conjectures:
\begin{enumerate}
\item[(i)] Every component of $I(f)$ is unbounded;
\item[(ii)] Every point of $I(f)$ can be connected to $\ity$ by a curve consisting of escaping points.
\end{enumerate}

For the exponential maps of the form $f(z)=e^z+\la$ with $\la>-1,$ it is known, by Rempe \cite{R2}, that the escaping set is a connected subset of the plane, and for $\la<-1,$ it is the disjoint union of uncountably many curves to infinity, each of which is connected component of $I(f)$ \cite{R4}. (Moreover, these maps have no critical points and exactly one asymptotic value which is the omitted value $\la$).
It was shown in ~\cite{SZ} that every escaping point of every exponential map can be connected to $\ity$ by a curve consisting of escaping points.
%------------------------------------------------------------------------------------------

Two functions $f$ and $g$ are called permutable if $f\circ g=g\circ f.$
Fatou  \cite{beardon}  proved that if $f$ and $g$ are two rational functions which are permutable, then $F(f)=F(g)$. This was the most famous result  that motivated the dynamics of composition of complex functions. Similar results for transcendental entire functions is still not known, though it holds in some very special cases   \cite[Lemma 4.5]{baker2}.
If $f$ and $g$ are transcendental entire functions, then so is $f\circ g$ and $g\circ f$ and the dynamics of one composite entire function helps in the study of the dynamics of the other and vice-versa. 
If $f$ and $g$ are transcendental entire functions, the dynamics of $f\circ g$ and $g\circ f$ are very similar. In  \cite{berg4, poon2'}, it was shown $f\circ g$ has  wandering domains if and only if $g\circ f$ has  wandering domains. In \cite{dinesh3}   the authors have constructed several  examples where the dynamical behavior of $f$ and $g$ vary greatly from the dynamical behavior of $f\circ g$ and $g\circ f.$ Using approximation theory of entire functions, the authors have shown the existence of entire functions $f$ and $g$ having infinite number of domains satisfying various properties and relating it to their compositions. They explored and enlarged all the maximum possible ways of the solution in comparison to the past result worked out. 

In \cite{dinesh2}, the authors considered the relationship between Fatou sets and singular values (to be defined in section \ref{sec2'})   of transcendental entire functions $f, g$ and $f\circ g$. They gave various conditions under which Fatou sets of $f$ and $f\circ g$ coincide and also considered relation between the singular values of $f, g$ and their compositions. 

A natural extension of the dynamics associated to the iteration of a complex function is the dynamics of composite of two or more such functions and this leads to the realm of semigroups of rational and transcendental entire functions. In this direction the seminal work was done by Hinkkanen and Martin \cite{martin} related to semigroups of rational functions. In their  paper, they extended the classical theory of the dynamics  associated to the iteration of a rational function of one complex variable to the more general setting of an arbitrary semigroup of rational functions. Many of the results were extended to semigroups of transcendental entire functions in \cite{cheng}, \cite{dinesh1}, \cite{poon1} and \cite{zhigang}.
In \cite{dinesh1}, the authors generalised  the dynamics of a transcendental entire function on its Fatou set to the dynamics of semigroups of transcendental entire functions. They also investigated the dynamics of conjugate semigroups, abelian transcendental semigroups and wandering and Baker domains of transcendental semigroups.   

A transcendental semigroup $G$ is a semigroup generated by a family of transcendental entire functions $\{f_1,f_2,\ldots\}$ with the semigroup operation being functional composition. We denote the semigroup by $G=[f_1,f_2,\ldots].$ Thus each $g\in G$ is a transcendental entire function and $G$ is closed under composition. For an introduction to iteration theory of entire functions, see \cite {berg1}.
%---------------------------------------------------------------------------------

A family $\F$ of meromorphic functions is normal in a domain $\Omega\subset\C$ if every sequence in $\F$ has a subsequence which converges locally uniformly in $\Omega$ to a meromorphic function or to the constant $\ity.$ 
The set of normality or the Fatou set $F(G)$ of a transcendental semigroup $G$, is the largest open subset of $\C$ on which the family of functions in $G$ is normal. Thus $z_0\in F(G)$ if it has a neighborhood $U$ on which the family $\{g:g\in G\}$ is  normal. The Julia set $J(G)$ of $G$ is the complement of $F(G),$ that is, $J(G)=\ti\C\setminus F(G).$ The semigroup generated by a single function $f$ is denoted by $[f].$ In this case we denote $F([f])$ by $F(f)$ and $J([f])$ by $J(f)$ which are the respective Fatou set and Julia set in the classical Fatou-Julia theory of iteration of a single transcendental entire function. The dynamics of a semigroup is more complicated than those of a single function. 
 For instance, $F(G)$ and $J(G)$ need not be completely invariant and $J(G)$ may not be the entire complex plane $\C$ even if $J(G)$ has an interior point  \cite{martin}.

 In this paper we have generalized  the dynamics of a transcendental entire function on its escaping set to the dynamics of semigroups of transcendental entire functions on their escaping set. For a transcendental semigroup $G$, we have initiated the study of escaping sets of semigroups of transcendental entire functions, to be denoted by $I(G).$ Some of the properties of the escaping set from the classical dynamics do not get preserved for the semigroups. For instance, $I(G)$ is only forward invariant and it could be an empty set also. Furthermore, the escaping sets of conjugate semigroups and abelian transcendental semigroups   have also been investigated. Several important classes of transcendental semigroups for which Eremenko's conjecture holds have been provided. Recall that the postsingular set of an entire function $f$ is defined as 
\[\mathcal P(f)=\overline{\B(\bigcup_{n\geq 0}f^n(\text{Sing}(f^{-1}))\B)}.\]
 The entire function $f$ is called \emph{postsingularly bounded} if $\mathcal P(f)$ is bounded. The function  $f$ is said to be \emph{postsingularly finite} if each of its singular value has a finite orbit, which means that each of the singular orbit is preperiodic and  $f$ is called \emph{hyperbolic} if the postsingular set $\mathcal P(f)$ is a compact subset of $F(f).$ Using these definitions, we have given definitions for \emph{postsingularly bounded}, \emph{postsingularly finite} and \emph{hyperbolic} transcendental semigroups.

\section{escaping set of a transcendental semigroup}\label{sec2}
We introduce the definition of escaping set for semigroups of transcendental entire functions.
\begin{definition}\label{sec2,defn'''''}
Let  $G=[g_1,g_2,\ldots]$  be a  transcendental semigroup. The escaping set of $G,$ denoted by $I(G)$ is defined as
\[I(G)=\{z\in\C |\,\text{every sequence in}\,G\,\text{has a subsequence which diverges to infinity at}\, z\}.\]
\end{definition}
%---------------------------------------------------------------
Thus $I(G)\subset I(g)$ for all $g\in G$ and so $I(G)\subset\bigcap_{g\in G}I(g).$ Since $\bigcap_{g\in G}I(g)$ can be an empty set, so for an arbitrary semigroup $G$ of transcendental entire functions, $I(G)$ could be an empty set. We illustrate this with a theorem, but we first prove two lemmas:

\begin{lemma}\label{sec2,lemab}
 Consider the two parameter family of functions $\mathcal F=\{f_{{\la},{\xi}}(z)= e^{-z+\la}+\xi: \la,\,\xi\in\C, \RE{\la}<0, \RE\,\xi\geq 1\}.$
Denote $f_{{\la},{\xi}}$ by $f$ itself for simplicity. Then for each $f\in\mathcal F,$ $\displaystyle{I(f)\subset\{z=x+iy : x<0, (4k-3)\frac{\pi}{2}<y<(4k-1)\frac{\pi}{2}, k\in\Z\}.}$
\end{lemma}
%---------------------------------------------

\begin{proof}
 Consider the set $S=\{z\in\C: \RE z >0, \RE(e^{-z+\la})>0,\,\la\in\C,\,\text{with}\,\RE\la<0\}.$ Observe that $\RE(e^{-z+\la})> 0,$ implies $\dfrac{(4k-1)\pi}{2}+\IM\la<y<\dfrac{(4k+1)\pi}{2}+\IM\la, k\in\Z.$ Thus the set $S$ is the entire right half plane $\{z:\RE z>0\}.$ We shall show that  no point in $S$ escapes to $\ity$ under iteration of each $f\in\mathcal F.$ For  this we show  $|f^k(z)|\leq 1+|\xi|$ for all $k\in\N,\, z\in S.$  Suppose on the contrary that  there exist $n\in\N$ and $z\in S$ such that $|f^n(z)|>1+|\xi|.$ Now $|f(f^{n-1}(z))|> 1+|\xi|,$  implies $1+|\xi|<|e^{-f^{n-1}(z)+\la}+\xi|.$ This shows that $e^{-\RE f^{n-1}(z)+\RE\la}>1$, and as $\RE\la<0,$ we obtain  $-\RE f^{n-1}(z)>0,$ that is, $\RE f^{n-1}(z)<0.$  Further this implies that $\RE(f(f^{n-2}(z)))<0,$ that is, $\RE(e^{-f^{n-2}(z)+\la}+\xi)<0,$ which implies  $-\RE(e^{-f^{n-2}(z)+\la})> \RE\xi\geq 1$ and since $|z|\geq -\RE z$ for all $z\in\C,$ we get $|(e^{-f^{n-2}(z)+\la})|>1,$ that is,  $e^{-\RE f^{n-2}(z)+\RE\la}>1$ and so  $\RE f^{n-2}(z)<0.$ By induction  we will get $\RE f(z)<0.$ But $\RE f(z)=\RE(e^{-z+\la})+\RE\xi>{0+1}=1,$  so we arrive at  a contradiction and therefore proves the assertion. Also no point on the imaginary axis $\{z: \RE z=0\}$ can escape to infinity under iteration of any $f\in\mathcal F.$ Moreover, the set $S_1=\{z\in\C: \RE z<0, \RE(e^{-z+\la})>0\}$ does not intersect $I(f)$ for any $f\in\mathcal F.$ Hence for each $f\in\mathcal F,$ 
$\displaystyle{I(f)\subset\{z=x+iy : x<0, (4k-3)\frac{\pi}{2}<y<(4k-1)\frac{\pi}{2}, k\in\Z\}.}$
\end{proof}
%---------------------------------------

\begin{lemma}\label{sec2,lemab'}
Consider the two parameter family of functions $\mathcal F'=\{f_{{\mu},{\ze}}(z)= e^{z+\mu}+\ze: \mu,\,\ze\in\C, \RE{\mu}<0, \RE\,\ze\leq-1\}.$
Denote $f_{{\mu},{\ze}}$ by $f$ itself for simplicity. Then for each $f\in\mathcal F,'$ $\displaystyle{I(f)\subset\{z=x+iy : x>0, (4k-1)\frac{\pi}{2}<y<(4k+1)\frac{\pi}{2}, k\in\Z\}.}$
\end{lemma}

\begin{proof}
 Consider the set $S'=\{z\in\C: \RE z <0, \RE(e^{z+\mu})<0,\,\mu\in\C\,\text{with}\,\RE\mu<0\}.$ Observe that $\RE(e^{z+\mu})< 0,$ implies $\dfrac{(4k-3)\pi}{2}-\IM\,\mu<y<\dfrac{(4k-1)\pi}{2}-\IM\mu, k\in\Z.$ Thus  the set $S'$ is the entire left half plane $\{z:\RE z<0\}.$ We shall show that no point in $S'$ escapes to $\ity$ under iteration of each $f\in\mathcal F'.$ For  this we show  $|f^k(z)|\leq 1+|\ze|$ for all $k\in\N,\,z\in S'.$  Suppose on the contrary that there exist $n\in\N$ and $z\in S'$ such that $|f^n(z)|>1+|\ze|.$
Now $|f(f^{n-1}(z))|> 1+|\ze|,$  implies $1+|\ze|<|e^{f^{n-1}(z)+\mu}+\ze|.$ This shows that $e^{\RE f^{n-1}(z)+\RE\mu}>1$, and as $\RE\mu<0,$ we obtain  $\RE f^{n-1}(z)>0.$   Further this implies that $\RE(f(f^{n-2}(z)))>0,$ that is, $\RE(e^{f^{n-2}(z)+\mu}+\ze)>0,$ which implies $\RE(e^{f^{n-2}(z)+\mu})>-\RE\ze\geq 1$  and since $|z|\geq \RE z$ for all $z\in\C,$ we get $|(e^{f^{n-2}(z)+\mu})|>1,$ that is,  $e^{\RE f^{n-2}(z)+\RE\mu}>1$ and so $\RE f^{n-2}(z)>0.$ By induction  we will get $\RE f(z)>0.$ But $\RE f(z)=\RE(e^{z+\mu})+\RE\ze<{0-1}=-1,$  so we arrive at  a contradiction and therefore proves the assertion. Also no point on the imaginary axis $\{z: \RE z=0\}$  can escape to infinity under iteration of any $f\in\mathcal F'.$ Moreover, the set $S_1'=\{z\in\C: \RE z>0, \RE(e^{z+\mu})<0\}$ does not intersect $I(f)$ for any $f\in\mathcal F'.$ Hence for each $f\in\mathcal F',$
$\displaystyle{I(f)\subset\{z=x+iy : x>0, (4k-1)\frac{\pi}{2}<y<(4k+1)\frac{\pi}{2}, k\in\Z\}.}$
\end{proof}
%-------------------------------------------------
\begin{remark}\label{sec2,rem1b}
The right half plane is invariant under each $f\in\mathcal F.$ The left half plane is invariant under each $f\in\mathcal F'.$
\end{remark}
%------------------------------------------------

\begin{theorem}\label{sec2,thmab}
 Consider the semigroup $G=[f,g]$, where $f(z)=e^{-z+\la}+\xi;\, \la,\xi\in\C,$\,with $\RE\la<0, \RE\xi\geq 1$ and $g(z)=e^{z+\mu}+\zeta; \,\mu,\zeta\in\C,$ \, with $\RE\mu<0, \RE\zeta\leq-1.$ Then $I(G)$ is an empty set.
\end{theorem}
%----------------------------------------------------------
\begin{proof}
From Lemmas \ref{sec2,lemab} and \ref{sec2,lemab'}, $\displaystyle{I(f)\subset\{z : x<0, (4k-3)\frac{\pi}{2}<y<(4k-1)\frac{\pi}{2}, k\in\Z\}}$ and $\displaystyle{I(g)\subset\{z : x>0, (4k-1)\frac{\pi}{2}<y<(4k+1)\frac{\pi}{2}, k\in\Z\}.}$ Thus $I(f)\cap I(g)=\emptyset$ which implies $I(G)=\emptyset.$
\end{proof}
There are several important classes of transcendental semigroups where $I(G)$ is non empty. These non trivial classes justify our study of escaping sets of transcendental semigroups. We next give some examples of transcendental semigroup $G$ whose escaping set $I(G)$ is non empty.
%-----------------------------------------------

\begin{example}\label{sec2,egi}
Let $f=e^{\la z},\,\la\in\C\setminus\{0\}$ and $g=f^s+\frac{2\pi i}{\la},\,s\in\N.$ Let $G=[f, g].$ For $n\in\N,\,g^n=f^{ns}+\frac{2\pi i}{\la}$ and so  $I(g)=I(f).$  Observe that for $ l,m,n,p\in\N, f^l \circ g^m=f^{l+ms}$ and $g^n\circ f^p=f^{ns+p}+\frac{2\pi i}{\la}$ and therefore for any $h\in G,$ either  $h=f^k$ for some $k\in\N,$ or $h=f^{qs}+\frac{2\pi i}{\la}=g^q$ for some $q\in\N.$ In either of the cases, $I(h)=I(f)$ and hence $I(G)=I(f)\neq\emptyset.$
\end{example}
%---------------------------------------

\begin{example}\label{sec2,egii}
Let  $f=\la\sin{z},\la\in\C\setminus\{0\}$ and $g=f^m+2\pi,\,m\in\N.$ Let $G=[f, g].$  On similar lines, it can be seen that\,$\emptyset\neq I(f)=I(G),$ and hence $I(G)\neq\emptyset.$
\end{example}
%------------------------------------------

\section{topological properties of escaping sets of certain classes of transcendental semigroups}\label{sec2'}
The following definitions are well known in  transcendental semigroup theory:

\begin{definition}\label{sec2,defn4'}
Let $G$ be a transcendental semigroup. A set $W$ is forward invariant under $G$ if $g(W)\subset W$ for all $g\in G$ and $W$ is backward invariant under $G$ if $g^{-1}(W)=\{w\in\C:g(w)\in W\}\subset W$ for all $g\in G.$ The set $W$ is called completely invariant under $G$ if it is both forward and backward invariant under $G.$ 
\end{definition}
%------------------------------------------

Recall that $w\in\C$ is a critical value of a transcendental entire function $f$ if there exist some $w_0\in\C$ with $f(w_0)=w$ and $f'(w_0)=0.$ Here $w_0$ is called a critical point of $f.$ The image of a critical point of $f$ is  critical value of $f.$ Also recall that $\zeta\in\C$ is an asymptotic value of a transcendental entire function $f$ if there exist a curve $\Gamma$ tending to infinity such that $f(z)\to \zeta$ as $z\to\ity$ along $\Gamma.$
The definitions for  critical point, critical value and asymptotic value of a transcendental semigroup $G$ were introduced in \cite{dinesh1}.

\begin{definition}\label{sec2,defn4''}
A point $z_0\in\C$ is called a critical point of $G$ if it is a critical point of some $g\in G.$  A point $w\in\C$ is called a critical value of $G$ if it is a critical value of some $g\in G.$
\end{definition}
%------------------------------------------------------------

\begin{definition}\label{sec2,defn4'''}
 A point $w\in\C$ is called an asymptotic value of $G$ if it is an asymptotic  value of some $g\in G.$
\end{definition}
%---------------------------------------

\begin{theorem}\label{sec2,thmi}
Let $f$ be a transcendental entire function  of period $c,$  let $g=f^s+c,$\,  $s\in\N$ and $G=[f, g].$ Then $F(G)=F(h)$ and $I(G)=I(h)$ for all $h\in G.$
\end{theorem}

\begin{proof}
For $n\in\N,\,g^n=f^{ns}+c$ and so $F(g)=F(f)$ and $I(g)=I(f).$  As in Example \ref{sec2,egi}, for $ l,m,n,p\in\N, f^l \circ g^m=f^{l+ms}$ and $g^n\circ f^p=f^{ns+p}+c$ and therefore for any $h\in G, h=f^k$ for some $k\in\N,$ or $h=f^{qs}+c=g^q$ for some $q\in\N.$ Thus $I(h)=I(f^k)=I(f)$ or $I(h)=I(g^q)=I(g)$ for all $h\in G$ and similarly $F(h)=F(f)$ or $F(h)=F(g)$ for all $h\in G.$ Hence $I(G)=I(h)$ and $F(G)=F(h)$ for all $h\in G.$
\end{proof}
%---------------------------------------------------

In \cite{SZ} it was shown that if $f$ is an exponential map, that is, $f= e^{\la z},\,\la\in\C\setminus\{0\},$ then all components of $I(f)$ are unbounded, that is, Eremenko's conjecture \cite{e1} holds for exponential maps. The above result gets generalized to a specific class of transcendental semigroup generated by exponential maps.

\begin{theorem}\label{sec2,thmiii}
Let $f=e^{\la z},\,\la\in\C\setminus\{0\},$ $g=f^m+\frac{2\pi i}{\la},\,m\in\N$ and $G=[f, g].$ Then all components of $I(G)$ are unbounded.
\end{theorem}

\begin{proof}
From Theorem \ref{sec2,thmi}, $I(G)=I(h)=I(f)$ for all $h\in G$ and so $I(G)=I(f).$ Since $f$ is an exponential map, so all components of $I(f)$ are unbounded, see \cite{SZ}, and therefore all components of $I(G)$ are unbounded. 
\end{proof}
%-----------------------------------------
Recall the
 Eremenko-Lyubich class
 \[\mathcal{B}=\{f:\C\to\C\,\,\text{transcendental entire}: \text{Sing}{(f^{-1})}\,\text{is bounded}\},\]
where Sing($f^{-1}$) is the set of critical values and asymptotic values of $f$ and their finite limit points. Each $f\in\mathcal B$ is said to be of bounded type. Moreover,  if $f$ and $g$ are of bounded type, then so is $f\circ g$ \cite{berg4}. 

Recall that an entire function $f$ is called \emph{postsingularly bounded} if the postsingular set $\mathcal P(f)$ is bounded. We now define postsingularly bounded in the context of semigroups.

\begin{definition}\label{sec2,defn1}
A transcendental semigroup $G$ is called \emph{postsingularly bounded} if each $g\in G$ is postsingularly bounded.
\end{definition}
%--------------------------------------
In \cite{R1} it was shown that if $f$ is an entire function of bounded type for which all singular orbits are bounded (that is, $f$ is postsingularly bounded), then each connected component of $I(f)$ is unbounded, providing a partial answer to a conjecture of Eremenko \cite{e1}. We  show  this result gets generalized to a particular class of postsingularly bounded semigroup. 

\begin{lemma}\label{sec2,lema}
Let $f\in\mathcal B$ be postsingularly bounded, then so is $f^k$ for all $k\in\N.$
\end{lemma}

\begin{proof}
For each $k\in\N, f^k$ is of bounded type. 
Let Sing${f^{-1}}=\{z_\la:\la\in\Lambda\}.$ Then Sing${(f^k)}^{-1}=\bigcup_{l=0}^{k-1}f^l(\text{Sing}f^{-1}),$ \cite{baker5}. We show for  $0\leq l\leq k-1$,  each singular value in $f^l(\text{Sing}f^{-1})$ has  bounded orbit under $f^k.$ A singular value in  $f^l(\text{Sing}f^{-1})$ has the form $f^l(z_\la)$ for some $z_\la\in \text{Sing}f^{-1}.$ By hypothesis, orbit of $z_\la$ under $f$ is bounded, that is, there exist $M_\la>0$ such that $|f^n(z_\la)|\leq M_\la$ for all  $n\in\N.$ Then  for all $n\in\N,$ \,$|f^{kn}(f^l(z_\la))|=|f^{kn+l}(z_\la)|\leq M_\la$  and hence the result.
\end{proof}
%--------------------------------------------
\begin{theorem}\label{sec2,thmi'}
Let $f\in\mathcal B$ be periodic of period $c$ and let $f$ be postsingularly bounded. Let $g=f^s+c, \,s\in\N$ and  $G=[f,g].$ Then $G$ is postsingularly bounded and all components of $I(G)$ are unbounded.
\end{theorem}
%------------------------------

\begin{proof} 
Observe that $g$ is of bounded type. 
We first show that $g$ is postsingularly bounded. Let Sing${f^{-1}}=\{z_\la:\la\in\Lambda\}.$ By hypothesis, for each $z_\la\in \text{Sing}f^{-1},$ \,there exist $M_\la >0$ such that  $|f^n(z_\la)|\leq M_\la$ for all $n\in\N.$ Now Sing\,$g^{-1}=\bigcup_{k=0}^{s-1}f^k (\text{Sing}{f^{-1}})+c.$ We show for $0\leq k\leq s-1,$ each singular value  of $f^k(\text{Sing}{f^{-1}})+c$ has bounded orbit  under $g$. A singular value in  $f^k(\text{Sing}f^{-1})+c$ has the form $f^k(z_\la)+c$ for some $z_\la\in \text{Sing}f^{-1}.$  Consider
\begin{equation}\label{sec2,eq1'}
\begin{split}\notag
g^n(f^k(z_\la)+c)
&=(f^s+c)^n(f^k(z_\la)+c)\\
&=f^{ns}(f^k(z_\la)+c)+c\\
&=f^{ns+k}(z_\la)+c.
\end{split}
\end{equation}
 As $|f^n(z_\la)|\leq M_\la$ for all $n\in\N,$ we obtain $|g^n(f^k(z_\la)+c)|\leq M_\la+|c|$ for all $n\in\N$ and hence $g$ is postsingularly bounded. Using the argument as in Theorem \ref{sec2,thmi}, any $h\in G$ has the form $h=f^l,$ for some $l\in\N,$ or $h=g^q,$ for some $q\in\N.$ Thus each $h\in G$ is of bounded type. Using Lemma \ref{sec2,lema}, $h$ is postsingularly bounded and hence $G$ is postsingularly bounded. From Theorem \ref{sec2,thmi}, we have $I(g)=I(f)$ and  $I(G)=I(h)=I(f)$ for all $h\in G.$ Applying the result \cite[Theorem 1.1]{R1}, all components of $I(f)$ are unbounded and hence all components of $I(G)$ are unbounded.
\end{proof}

We illustrate this with an example:

\begin{example}\label{sec2,egab'''}
Let $f=e^{\la z},\,\la>\frac{1}{e},$ $g=f^s+\frac{2\pi i}{\la},\,s\in\N$ and  $G=[f,g].$ Suppose that the orbit of Sing{$f^{-1}=\{0\}$} is bounded.
Then there exist $M>0$ such that $|f^n(0)|\leq M$ for all $n\in\N.$ Now Sing\,$g^{-1}=\bigcup_{k=0}^{s-1}f^k (\text{Sing}{f^{-1}})+\frac{2\pi i}{\la}.$ 
 For each $0\leq k\leq s-1,$
\begin{equation}\label{sec2,eq1}
\begin{split}\notag
g^n\B(f^k(0)+\frac{2\pi i}{\la}\B)
&=\B(f^s+\frac{2\pi i}{\la}\B)^n\B(f^k(0)+\frac{2\pi i}{\la}\B)\\
&=f^{ns}\B(f^k(0)+\frac{2\pi i}{\la}\B)+\frac{2\pi i}{\la}\\
&=f^{ns+k}(0)+\frac{2\pi i}{\la}.
\end{split}
\end{equation}
 As $|f^n(0)|\leq M$ for all $n\in\N,$ we obtain $|g^n(f^k(0)+\frac{2\pi i}{\la})|\leq M+\frac{2\pi }{\la},$ for all $n\in\N$ and  for $0\leq k\leq s-1,$ which implies   $g$ is postsingularly bounded. Using the argument as in Theorem \ref{sec2,thmi}, any $h\in G$ has the form $h=f^l,$ for some $l\in\N,$ or $h=g^q,$ for some $q\in\N.$ Using Lemma \ref{sec2,lema}, $h$ is postsingularly bounded and hence $G$ is postsingularly bounded. From Theorem \ref{sec2,thmi}, we have $I(g)=I(f)$ and  $I(G)=I(h)=I(f)$ for all $h\in G.$ Applying  result \cite[Theorem 1.1]{R1}, all components of $I(f)$ are unbounded and hence all components of $I(G)$ are unbounded.
\end{example}
%------------------------------------------

\begin{remark}\label{sec2,rem1}
In above example, as $\la>\frac{1}{e},$ therefore $I(f)$ is connected \cite{R4}. Moreover,  $I(h)$ is connected for all $h\in G$  and as $I(G)=I(f),$ $I(G)$ is also connected.
\end{remark}
%--------------------------------------------
\begin{remark}\label{sec2,rem2}
In above example, $F(f)=\emptyset,$ \cite{dev1} and hence $F(G)=\emptyset.$
\end{remark}
%------------------------------------------

%-------------------------------------------------
Recall that an entire function $f$ is called \emph{postsingularly finite} if each of its singular value has a finite orbit, which means that each of the singular orbit is preperiodic. 
%The map $f$ is sometimes called \emph{`postsingularly preperiodic'}.
We now define postsingularly finite in the context of semigroups.
%---------------------------------------------------------
\begin{definition}\label{sec2,defn3}
A transcendental semigroup $G$ is called \emph{postsingularly finite} if each $g\in G$ is postsingularly finite.
\end{definition}
%-----------------------------------------------
It follows from  result \cite[Theorem 1.1]{R1} that if $f\in\mathcal B$ is postsingularly finite, then all components of $I(f)$ are unbounded. We shall show that this result gets generalized to a specific class of postsingularly finite semigroup.

%-----------------------------------------------------
\begin{lemma}\label{sec2,lemb}
Let $f\in\mathcal B$ be postsingularly finite, then so is $f^k$ for all $k\in\N.$
\end{lemma}
%-----------------------------------------
\begin{proof}
Let Sing${f^{-1}}=\{z_\la:\la\in\Lambda\}.$ Then Sing${(f^k)}^{-1}=\bigcup_{l=0}^{k-1}f^l(\text{Sing}f^{-1}).$ We show for each $0\leq l\leq k-1$,  each singular value in $f^l(\text{Sing}f^{-1})$ has  finite orbit under $f^k.$ A singular value in  $f^l(\text{Sing}f^{-1})$ has the form $f^l(z_\la),$ for some $z_\la\in \text{Sing}f^{-1}.$ By hypothesis, the orbit of $z_\la$ is preperiodic under $f.$ So  for some $n\in\N,$ $f^n(z_\la)$ is periodic  of period say $p,$ that is, $f^p(f^n(z_\la))=f^n(z_\la).$ Then 
\begin{equation}\label{sec2,eqz'}
\begin{split}\notag
(f^k)^{(p+n)}(f^l(z_\la))
&=f^l (f^{k(p+n)}(z_\la))\\
&=f^l (f^{kn}(z_\la))\\
&=f^{kn}(f^l(z_\la)),
\end{split}
\end{equation}
 which implies $f^{kp}(f^{kn}(f^l(z_\la)))=f^{kn}(f^l(z_\la)).$ Thus $f^l(z_\la)$ is preperiodic under $f^k$ and hence the result.
\end{proof}
%--------------------------------------------------
\begin{theorem}\label{sec2,thmiv}
Let $f\in\mathcal B$ be  of period $c$ and let $f$ be postsingularly finite. Let $g=f^s+c$,\, $s\in\N$ and $G=[f,g].$ Then $G$ is postsingularly finite and all components of $I(G)$ are unbounded.
\end{theorem}
%--------------------------------------------------
\begin{proof}
We first show that $g$ is postsingularly finite. Let Sing${f^{-1}}=\{z_\la:\la\in\Lambda\}.$ Now Sing\,$g^{-1}=\bigcup_{k=0}^{s-1}f^k (\text{Sing}{f^{-1}})+c.$ We show for $0\leq k\leq s-1,$ each singular value  of $f^k(\text{Sing}{f^{-1}})+c$ is preperiodic under $g$. A singular value in  $f^k(\text{Sing}f^{-1})+c$ has the form $f^k(z_\la)+c,$ for some $z_\la\in \text{Sing}f^{-1}.$  By hypothesis, orbit of $z_\la$ is preperiodic under $f.$ So  for some $n\in\N,$ $f^n(z_\la)$ is periodic  of period say $p,$ that is, $f^p(f^n(z_\la))=f^n(z_\la).$ Then
\begin{equation}\label{sec2,eqz''}
\begin{split}\notag
 g^{(p+n)}(f^k(z_\la)+c)
&=(f^s+c)^{(p+n)} (f^k(z_\la)+c)\\
&=f^{s(p+n)}(f^k(z_\la)+c)+c\\
&=f^k (f^{s(p+n)}(z_\la))+c\\
&=f^k (f^{sn}(z_\la))+c\\
&=f^{sn}(f^k(z_\la))+c\\
&=(f^s+c)^n(f^k(z_\la)+c)\\
&=g^n(f^k(z_\la)+c),
\end{split}
\end{equation}
 which implies $g^p( g^n(f^k(z_\la)+c))=g^n(f^k(z_\la)+c).$ Thus $f^k(z_\la)+c$ is preperiodic under $g$ which implies $g$ is postsingularly finite. Using the argument as in Theorem \ref{sec2,thmi}, any $h\in G$ has the form $h=f^l,$ for some $l\in\N,$ or $h=g^q,$ for some $q\in\N.$ Using Lemma \ref{sec2,lemb}, $h$ is postsingularly finite and hence $G$ is postsingularly finite. From Theorem \ref{sec2,thmi}, we have $I(g)=I(f)$ and  $I(G)=I(h)=I(f)$ for all $h\in G.$ Applying result \cite[Theorem 1.1]{R1}, all components of $I(f)$ are unbounded and hence all components of $I(G)$ are unbounded.
\end{proof}
%-------------------------------------------------------
%%-------------------------------------------------------
Recall that an entire function $f$ is called \emph{hyperbolic} if the postsingular set $\mathcal P(f)$ is a compact subset of $F(f).$ For instance, $e^{\la z},\,0<\la<\frac{1}{e}$ are examples of hyperbolic entire functions. We next define hyperbolic in the context of  semigroups.
%---------------------------------------
\begin{definition}\label{sec2,defn4}
A transcendental semigroup $G$ is called \emph{hyperbolic} if each $g\in G$ is hyperbolic.
\end{definition}
%-----------------------------------------------

It follows from  result \cite[Theorem 1.1]{R1} that if $f\in\mathcal B$ is hyperbolic, then all components of $I(f)$ are unbounded. We  show  this result gets generalized to a particular class of hyperbolic semigroup.
In \cite{baker5}, it was shown that $\mathcal P(f^k)=\mathcal P(f)$ for all $k\in\N.$
Therefore if $f$ is hyperbolic, then so is $f^k$ for all $k\in\N.$
%-------------------------------------------------
\begin{theorem}\label{sec2,thmv}
Let $f\in\mathcal B$ be  of period $c$ and let $f$ be hyperbolic. Let $g=f^m+c$,\, $m\in\N$ and $G=[f,g].$ Then $G$ is hyperbolic and all components of $I(G)$ are unbounded.
\end{theorem}

\begin{proof}
We first show that $\mathcal P(g)=\mathcal P(f)+c.$ Consider
\begin{equation}\label{sec2,eq2}
\begin{split}\notag
\mathcal P(g)
&=\bigcup_{n\geq 0}g^n(\text{Sing}\,g^{-1})\\
&=\bigcup_{n\geq 0}\bigcup_{k=0}^{m-1}(f^m+c)^n(f^k(\text{Sing}f^{-1})+c)\\
&=\bigcup_{n\geq 0}\bigcup_{k=0}^{m-1}(f^{mn}(f^k(\text{Sing}f^{-1})+c)+c)\\
&=\bigcup_{n\geq 0}\bigcup_{k=0}^{m-1}(f^{mn}(f^k(\text{Sing}f^{-1})+c))\\
&=\bigcup_{n\geq 0}(f^n(\text{Sing}f^{-1})+c)\\
&=\mathcal P(f)+c.
\end{split}
\end{equation}
Since $f$ is hyperbolic so is $g.$
  Using the argument as in Theorem \ref{sec2,thmi}, any $h\in G$ has the form $h=f^l,$ for some $l\in\N,$ or $h=g^q,$ for some $q\in\N.$ From observation made above, $h$ is hyperbolic and hence $G$ is hyperbolic. From Theorem \ref{sec2,thmi}, we have $I(g)=I(f)$ and  $I(G)=I(h)=I(f)$ for all $h\in G.$ Applying result \cite[Theorem 1.1]{R1}, all components of $I(f)$ are unbounded and hence all components of $I(G)$ are unbounded.
\end{proof}
%--------------------------------------

\section{the escaping set and the julia set}\label{sec3}
Throughout this section we assume $I(G)$ to be non empty. 
An important property of the escaping set of a transcendental semigroup $G$ is its forward invariance under $G$

\begin{theorem}\label{sec2,thmvi}
Let $G=[g_1,g_2,\dots]$ be a transcendental semigroup. Then $I(G)$ is forward invariant under $G.$
\end{theorem}

\begin{proof}
Let $z_0\in I(G).$ Then every sequence $\{g_n\}$ in $G$ has a subsequence say, $\{g_{n_k}\}$ which diverges to $\ity$ at $z_0.$ Consider the sequence $\{g_n\circ g : g\in G\}$ in $G.$ Then it has a subsequence $\{g_{n_k}\circ g: g\in G\}$ which diverges to $\ity$ at $z_0,$ that is, $g_{n_k}(g(z_0))\to\ity$ as $k\to\ity$ for all $g\in G.$ This implies $g(z_0)\in I(G)$ for all $g\in G$ and hence $I(G)$ is forward invariant.
\end{proof}
%-------------------------------------------
For a transcendental entire function $f$,\,the Julia set of $f$ is the boundary of escaping set of $f$ \cite{e1}. We show this result gets generalized to  transcendental semigroups. We first prove a lemma

\begin{lemma}\label{sec2,lem2}
Let $G=[g_1,g_2,\ldots]$ be a transcendental semigroup. Then
\begin{enumerate}\item[(i)] $int(I(G))\subset F(G)$ where $int (S)$ denotes interior of a subset $S$ of a topological space $X$;
\item[(ii)] $ext(I(G))\subset F(G)$ where $ext (S)$ denotes exterior of a subset $S$ of a topological space $X$ and is defined as $ext (S)=int (X\setminus S).$
\end{enumerate}
\end{lemma}

\begin{proof}
\begin{enumerate}\item[(i)] Suppose $\ze\in int(I(G)).$ Then there exist a neighborhood $U$ of $\ze$ such that $\ze\in U\subset I(G).$ As $I(G)$ is forward invariant under $G$, for every sequence $\{g_n\}$ in $G$, $g_n(U)$ never hits a non escaping point. Since there are at least two preperiodic points, so $\{g_n\}$ is normal on $U$ by Montel's theorem and so $U\subset F(G)$ and hence the result.

\item[(ii)] Suppose $z_0\in ext(I(G))\cap J(G).$ Then there exist a neighborhood $V$ of $z_0$ such that $z_0\in V\subset ext(I(G)).$ Also from \cite[Theorem 4.2]{poon1},\,$J(G)=\overline{(\bigcup_{g\in G}J(g))}$, and as $z_0\in J(G),$  there exist a sequence $\{g_j\}\subset G$ and for each $j\in\N,$ $\ze_j\in J(g_j)$ such that $\ze_j\to z_0.$ There exist $N>0$ such that $\ze_j\in V\,$ for all $j\geq N.$ In particular $\ze_N\in V$ and $\ze_N\in J(g_N).$ By blowing up property of $J(g_N),$\,$\bigcup_{n\in\N}g_{N}^{n}(V)$ can omit at most one point of $\C$. But this is a contradiction, since $V\subset ext(I(G))$ and so $(g_{N}^{n}\vert_{V}:n\in\N)$ is bounded and hence the result.\qedhere
\end{enumerate}
\end{proof}

%%---------------------------------------------------------

\begin{theorem}\label{sec2,thm7}
Let $G=[g_1,g_2,\ldots]$ be a transcendental semigroup. Then $J(G)=\partial I(G)$ where $\partial S$ denotes boundary of a set $S.$
\end{theorem}
%--------------------------------------------
\begin{proof}
$J(G)\subset\partial I(G)$ follows from above Lemma \ref{sec2,lem2}, and $\partial I(G)\subset J(G)$ is evident.
\end{proof}
%--------------------------------------
The following corollary is immediate
\begin{corollary}\label{sec2,cor1}
Let $G=[g_1,g_2,\ldots]$ be a transcendental semigroup. Then $J(G)\subset\overline{I(G)}.$
\end{corollary}
%----------------------------------------

For a transcendental entire function $f$ of bounded type, $I(f)\subset J(f)$ \cite{el2}. Also for a transcendental entire function $f$ of bounded type \,$J(f)=\overline{I(f)}$ \cite{el2}. We show  these results get generalized to a finitely generated transcendental semigroup in which the generators are of bounded type.

\begin{theorem}\label{sec2,thm5}
Let $G=[g_1,\ldots,g_n],\,g_i\in\mathcal B,\,1\leq i\leq n$ be a finitely generated transcendental semigroup. Then
\begin{enumerate}
\item[(i)] $I(G)\subset J(G);$
\item[(ii)]  $J(G)=\overline{I(G)}.$
\end{enumerate}
\end{theorem}

\begin{proof}
\begin{enumerate} \item[(i)] Observe that, each  $g\in G$ is of bounded type and so $I(g)\subset J(g)$ for all $g\in G.$ Also  from \cite[Theorem 4.2]{poon1},\,$J(G)=\overline{(\bigcup_{g\in G}J(g))}$.
 Hence for each $g\in G,\,I(G)\subset I(g)\subset J(g)\subset J(G).$
\item[(ii)] This is evident from part\,(i) of this theorem and Corollary \ref{sec2,cor1}.\qedhere
\end{enumerate}
\end{proof}

 If $G$ is a transcendental semigroup and $U$  a multiply connected component of $F(G),$ define 
\[
\ti\G_U=\{g\in G: F(g)\,\text{has a multiply connected component}\, \ti U_g\supset U\}
\] as in \cite[Definition 3.1]{dinesh1}.
%----------------------------------------------------------

For an entire function $f$ having a multiply connected component, it was shown in \cite[Theorem 2]{RS1} that $I(f)$ is connected. We now show that for a transcendental semigroup $G$ which has a multiply connected component $U,$ each element in $\ti\G_U$ has connected escaping set.

\begin{theorem}\label{sec2,thm9}
Let $G=[g_1,g_2,\ldots]$ be a transcendental semigroup such that $F(G)$ has a multiply connected component $U.$ Then for all $g\in\ti\G_U,\, I(g)$ is connected.
\end{theorem}
%--------------------------------------
\begin{proof}
For all $g\in\ti\G_U,\,U\subset\ti U_g,$ where $\ti U_g\subset F(g)$ is multiply connected. Using result  \cite[Theorem 2]{RS1}, $I(g)$ is connected for all $g\in\ti\G_U.$
\end{proof}
%-----------------------------
Recall that a finitely generated transcendental semigroup $G=[g_1,\ldots,g_n]$ is called abelian if $g_i\circ g_j=g_j\circ g_i,\,$ for all $1\leq i,j\leq n.$ 
For a transcendental entire function $f,\,\overline{I(f)}$ has no bounded components  \cite{e1}.
We show  in the following cases $\overline{I(G)}$ has no bounded components:

\begin{theorem}\label{sec2,thm10}
Let $G=[g_1,\ldots,g_n],$ $g_i\in\mathcal B,\,1\leq i\leq n$ be a finitely generated abelian transcendental semigroup. Then $\overline{I(G)}$ has no bounded components.
\end{theorem}
%----------------------------------------------

\begin{proof}
Observe that each  $g\in G$ is of bounded type and so $J(g)=\overline{I(g)}$ for all $g\in G,$ \cite{el2}.  
From \cite[Theorem 5.9]{dinesh1}, $J(G)=J(g)$ for all $g\in G.$ Also from Theorem \ref{sec2,thm5}(ii), $J(G)=\overline{I(G)},$ we have $\overline{I(G)}=\overline{I(g)}$ for all $g\in G.$ Since $\overline{I(g)}$ has no bounded components \cite{e1}, so does $\overline{I(G)}.$
\end{proof}
%--------------------------------------------------------------

\begin{theorem}\label{sec2,thm10'}
Let $g_1=f$ be a transcendental entire function of period $c.$ For $n>1,\,$ let  $g_n=f^{n-1}+c$  and $G=[g_1,\ldots,g_n].$ Then $\overline{I(G)}$ has no bounded components.
\end{theorem}

\begin{proof}
On similar lines to Theorem \ref{sec2,thmi}, $I(G)=I(g)$ for all $g\in G$ and hence all components of $\overline{I(G)}$ are unbounded.
\end{proof}

Recall \cite[Section 3]{dinesh1}, for an invariant component $U\subset F(G),$ the \emph{stabilizer} of $U$ is the set $G_U=\{g\in G:U_g=U\},$ where  $U_g$ denotes  the component of $F(G)$ containing $g(U)$. The next result gives a criterion for an invariant component of $F(G)$ to be a Siegel disk.

\begin{theorem}\label{sec2,thm8}
Let $G=[g_1\ldots,g_n]$ be a finitely  generated transcendental semigroup and let $U\subset F(G)$ be an invariant component. If $G_U$ has some non constant limit function on $U$, that is, every sequence in $G_U$ has a subsequence which converges locally uniformly on $U$ to some non constant limit function, then $U$ is a Siegel disk and this implies $U$ is  simply connected.
\end{theorem}
%-------------------------------------

\begin{proof}
As $U\subset F(G)$ is an invariant component, $U$ is periodic (of period $1$). From classification of periodic components of $F(G)$ \cite[Section 3]{dinesh1}, for all $g\in G_U,$ $U$ is either contained in an attracting domain, or a parabolic domain, or a  Siegel disk or a Baker domain of $F(g).$ As the sequence $\{g^n\},\,g\in G_U$ has a non constant limit function on $U,$ therefore by normality, $U$ must be contained in a Siegel disk $U^g\subset F(g)$ for all $g\in G_U$ (the other possibilities are ruled out, see  \cite[p.\ 64]{Hua}). This implies $U$ must itself be a Siegel disk of $F(G)$ and therefore simply connected.
\end{proof}
%-------------------------------------------

We now deal with conjugate semigroups. Recall that two entire functions $f$ and $g$ are conjugate if there exist a conformal map $\phi:\C\to\C$ with $\phi\circ f=g\circ\phi.$ By a conformal map $\phi:\C\to\C$ we mean an analytic and univalent map of the complex plane $\C$ that is exactly of the form $az+b,$ for some non zero $a.$
\begin{definition}\label{sec4,defn1}
Two finitely generated semigroups $G$ and $G'$  are said to be conjugate under a conformal map $\phi:\C\to\C$ if 
\begin{enumerate}
\item\ they have same number of generators,
\item\ corresponding generators are conjugate under $\phi.$
\end{enumerate}
If $G=[g_1,\ldots,g_n]$ is a finitely generated transcendental semigroup, we represent the conjugate semigroup $G'$ of $G$ by $G'=[\phi\circ g_1\circ\phi^{-1},\ldots,\phi\circ g_n\circ\phi^{-1}],$ where $\phi:\C\to\C$ is the conjugating map. Furthermore, if $G$ is abelian and each of its generators $g_i, 1\leq i\leq n,$ are of bounded type, then so is $G'.$
\end{definition}
%--------------------------------------------------------------------
If $f$ and $g$ are two rational functions which are conjugate under some Mobius transformation $\phi:\ti\C\to\ti\C$, then it is well known \cite[p.\ 50]{beardon}, $\phi(J(f))=J(g).$ Let $f$ be an entire function of bounded type which is conjugate under the conformal map $\phi$ to an entire function $g.$ Then $g$ is also of bounded type and moreover,  $\phi(\overline{I(f)})=\overline{I(g)}.$ We now show the closures  of $I(G)$ and $I(G')$ are similar for a finitely generated abelian transcendental semigroup in which each generator is of bounded type.

\begin{theorem}\label{sec3,thm11}
Let $G=[g_1,\ldots, g_n]$ be a finitely generated abelian transcendental semigroup in which each $g_i, 1\leq i\leq n$ is of bounded type and let $G'=[\phi\circ g_1\circ\phi^{-1},\ldots,\phi\circ g_n\circ\phi^{-1}]$ be the conjugate semigroup, where $\phi:\C\to\C$ is the conjugating map. Then $\phi(\overline{I(G)})=\overline{I(G')}.$
\end{theorem}
%----------------------------------------------
\begin{proof}
Denote for each $i, 1\leq i\leq n,\,\phi\circ g_i\circ\phi^{-1}$ by $g_i'.$ From \cite[Theorem 5.9]{dinesh1} and Theorem \ref{sec2,thm5}, $\overline{I(G)}=\overline{I(g_i)},\, 1\leq i\leq n.$ Moreover, as $G'$ is abelian and each of its generators are of bounded type, using the same argument again  implies $\overline{I(G')}=\overline{I(g_i')},\, 1\leq i\leq n.$ Thus
\begin{equation}
\begin{split}
\notag
\phi(\overline{I(G)})
&=\phi(\overline{I(g_i)})\\
&=\overline{I(g_i')}\\
&=\overline{I(G')}.\qedhere
\end{split}
\end{equation}
\end{proof}
%-----------------------------------------------

Acknowledgements. We are thankful to Prof. Kaushal Verma, IISc Bangalore for his valuable comments and suggestions. We are also grateful to the referee's  diligence and helpful suggestions to enhance the quality of the paper.

%\section*{\refname}

\end{document}